\documentclass[11pt]{article}

\usepackage[utf8]{inputenc}
\usepackage[T1]{fontenc}
\usepackage{amsmath,amssymb,amsthm}
\usepackage{mathrsfs}
\usepackage{geometry}
\usepackage{enumitem}
\usepackage{hyperref}
\usepackage{xcolor}
\colorlet{extrablue}{blue!60!black}
\newcommand{\extra}[1]{{\color{extrablue}#1}}
\usepackage{comment}
\theoremstyle{plain}
\newtheorem{theorem}{Theorem}[section]
\newtheorem{corollary}[theorem]{Corollary}
\newtheorem{proposition}[theorem]{Proposition}
\newtheorem{lemma}[theorem]{Lemma}
\theoremstyle{definition}
\newtheorem{definition}[theorem]{Definition}
\theoremstyle{remark}
\newtheorem{remark}[theorem]{Remark}

\newcommand{\OK}{\mathcal{O}_K}
\newcommand{\bb}{\mathfrak{b}}
\newcommand{\bB}{\mathfrak{B}}
\newcommand{\htop}{h_{\mathrm{top}}}
\newcommand{\Z}{\mathbb{Z}}
\newcommand{\Q}{\mathbb{Q}}
\newcommand{\Ftheta}{\mathscr{F}_\theta}

\title{Intrinsic ergodicity for $\bB$-free integers in number fields}
\author{Francesco Cellarosi}
\date{}

\begin{document}
\maketitle
\begin{abstract}
Let $K$ be a number field with ring of integers $\OK$, and let $\bB$ be an
Erd\H{o}s family of ideals in $\OK$. We prove that the associated $\bB$-free
subshift $(X_{\bB},(S_a)_{a\in\OK})$ is intrinsically ergodic: it carries a unique
measure of maximal entropy, which we identify explicitly as a relatively
independent extension of the Haar rotation on $\prod_{\bb\in\bB}\OK/\bb$. This is
the first proof of intrinsic ergodicity for $\bB$-free systems beyond dimension
one, and relies on the work of Ara\'ujo--Dymek--Ku\l aga-Przymus. Via their reductions, we  also settle
the $k$-free and $\bB$-free lattice-point cases and the $k$-free number-field case.
We give two independent proofs of the underlying rigidity statement: one
through a single-site relative-entropy argument, and one through an
exact-tiling realisation of Peckner's induce-and-split scheme.
\end{abstract}

\section{Introduction}\label{sec:intro}

\subsection{Sarnak's program and the squarefree subshift}

The M\"obius function $\mu$ and its square $\mu^2$, the indicator of the
squarefree integers, have been the focus of Sarnak's program~\cite{Sa1,Sa2}. Extending $\mu^2$
symmetrically and taking the orbit closure of the resulting point under the left
shift produces the \emph{squarefree subshift} $(X_{\mu^2},S)\subseteq\{0,1\}^\Z$,
and Sarnak proposed to establish for it a list of structural properties:
genericity of $\mu^2$ (the existence of all pattern frequencies) for a
shift-invariant measure of zero Kolmogorov entropy, nowadays called the
\emph{Mirsky measure} of the subshift; a formula for the topological entropy; a description of the
subshift in terms of admissible sequences; proximality with a unique minimal
set; and a non-trivial joining with a rotation despite a trivial maximal
equicontinuous factor.

These were proved by several authors. Cellarosi and Sinai~\cite{CS} established
the genericity statement, showing moreover that the subshift equipped with its
Mirsky measure is isomorphic to a rotation on
the compact group $\prod_p\Z/p^2\Z$ (in particular, of zero entropy). The
topological entropy was found to equal $6/\pi^2$, and the admissibility
description was given, by Peckner~\cite{Pe}, while the proximality and joining
statements were settled by Huck and Baake~\cite{HB}.

Crucially for the present paper, Peckner~\cite{Pe} proved a property that lies
beyond this list: the squarefree subshift is
\emph{intrinsically ergodic}, i.e.\ it carries a \emph{unique} measure of
maximal entropy. This is a genuine rigidity statement. The squarefree subshift
is hereditary and of positive topological entropy, hence supports a large
simplex of invariant measures; the Mirsky measure, which makes the system isomorphoic to a
rotation, has zero entropy and is therefore very far from maximal.
Intrinsic ergodicity asserts that, among all invariant measures, exactly one
attains the entropy maximum, and so singles out a canonical ``most chaotic''
measure for the system.

\subsection{$\bB$-free integers in number fields and  $\bB$-free systems}

The squarefree integers are the prototype of a \emph{$\bB$-free set}. Given
$\bB\subseteq\Z$, one removes from $\Z$ all multiples of elements of $\bB$; under the
\emph{Erd\H{o}s} hypothesis ($\bB$ infinite, pairwise coprime, and
$\sum_{b\in\bB}1/b<\infty$) the resulting set $F_{\bB}$ has a density and the
indicator $\eta=\mathbf 1_{F_{\bB}}$ generates a well-behaved subshift $(X_{\bB},S)$. In
this one-dimensional setting, El~Abdalaoui, Lema\'nczyk and de~la~Rue~\cite{ALR}
proved the analogues of the genericity, entropy, and admissibility statements; Ku\l aga-Przymus, Lema\'nczyk and
Weiss~\cite{KLW} proved intrinsic ergodicity and gave a full description of the
invariant measures; and Dymek, Ku\l aga-Przymus and Sell~\cite{DKS} revisited the
theory and extended intrinsic ergodicity to \emph{every} $\bB\subseteq\Z$, well
beyond the Erd\H{o}s case. The broader topological dynamics of $\bB$-free systems
(proximality, entropy, admissibility, equicontinuous factors) is developed
in~\cite{DKKL,Kw} and the references therein.

Several multidimensional and arithmetic generalisations followed. Pleasants and
Huck~\cite{PH} and Huck and Baake~\cite{HB} studied \emph{$k$-free lattice
points} in $\Z^d$; Baake and Huck~\cite{BH} announced the corresponding theory
for \emph{$\bB$-free lattice points}; and Cellarosi and Vinogradov~\cite{CV}
treated \emph{$k$-free integers in number fields} (with further recent developmen in the quadratic case by Baake-Bustos-Nickel~\cite{BBN} and Baake-Luz-Schindler~\cite{BLS}.

Recently, Ara\'ujo,
Dymek and Ku\l aga-Przymus~\cite{ADK} unified all of these under the heading of
\emph{$\bB$-free integers in number fields}, and Ara\'ujo~\cite{Ar1,Ar2} has
extended the Mirsky-measure side of the theory further still, to arbitrary
Erd\H{o}s sieves over \'etale $\Q$-algebras.
Our work shares the set-up of~\cite{ADK}, which we now describe. 

Let $K$ be a number field of degree $d=[K:\Q]$ with ring of integers $\OK$. The space $\{0,1\}^{\OK}$ (equipped with the product topology) is compact and
metrisable. For $x\in\{0,1\}^{\OK}$ we write $x_a:=x(a)$ for its coordinates and
put $\operatorname{supp}x:=\{a\in\OK:x_a=1\}$. The group $\OK\cong \Z^d$ acts on the compact topological space $\{0,1\}^{\OK}$ via shifts as  
\begin{align}\label{eq:OK-action}
(S_ax)(b):=x(a+b),\qquad a,b\in\OK,
\end{align}
so that each $S_a$ is a homeomorphism and $S_a S_{a'}=S_{a+a'}$. 
For a
nonzero ideal $\bb\subseteq\OK$ the quotient $\OK/\bb$ is finite, and
$N(\bb):=|\OK/\bb|$ denotes the norm of $\bb$. Two ideals
$\bb,\mathfrak c$ are \emph{coprime} if $\bb+\mathfrak c=\OK$.

\begin{definition}[Erd\H{o}s family]\label{def:erdos}
A family $\bB=\{\bb_\ell:\ell\ge1\}$ of proper nonzero ideals of $\OK$ is
\emph{Erd\H{o}s} if and only if it is infinite, the $\bb_\ell$'s are pairwise coprime, and
$\sum_{\ell\ge1}1/N(\bb_\ell)<\infty$.
\end{definition}
Given a family $\bB$ of ideals of $\OK$, the set of $\bB$-free integers is $F_{\bB}=\OK\setminus\bigcup_{\bb\in\bB}\bb$, and $\eta=\mathbf 1_{F_{\bB}}\in\{0,1\}^{\OK}$ denotes its indicator.
The orbit
closure of $\eta$ in $\{0,1\}^{\OK}$ is $X_\eta:=\overline{\{S_a\eta:a\in\OK\}}$.
Henceforth, we assume that $\bB$ is an Erd\H{o}s family of ideals. In this context, the authors of \cite{ADK} 
proved full analogues of the properties in  Sarnak's program; in particular $\eta$ is
generic for its Mirsky measure $\nu_\eta$ (Theorem A in \cite{ADK}). 
For each $\ell\geq1$, we equip the finite abelian group $\OK/\bb_\ell$
with its normalised counting measure. We set
\begin{equation}\label{eq:def-GP}
G:=\prod_{\ell\ge1}\OK/\bb_\ell,\qquad P:=\bigotimes_{\ell\ge1}(\text{normalised counting measure on }\OK/\bb_\ell),
\end{equation}
so that $G$ is a compact metrisable abelian group and $P$ is its Haar measure.
Let $\pi_\ell\colon\OK\to\OK/\bb_\ell$ be the quotient homomorphism. 
The group $\OK$ acts on $G$ by the rotation
\begin{equation}\label{eq:def-rotation}
T_a(g):=
(g_\ell+\pi_\ell(a))_{\ell\geq1},\qquad a\in\OK,\ g=(g_\ell)_{\ell\geq1}\in G.
\end{equation}
For brevity we abbreviate $g_\ell+\pi_\ell(a)$ as $g_\ell+a\in\OK/\bb_\ell$. We also use the shorthands 
\begin{equation}\label{eq:S-and-T}
S=(S_a)_{a\in\OK}\quad \mbox{and}\quad T=(T_a)_{a\in\OK},
\end{equation}
to refer to the actions \eqref{eq:OK-action} and \eqref{eq:def-rotation}, respectively. 
The measure-theoretic subshift $(X_\eta,\nu_\eta, S)$ 
is isomorphic to the Haar rotation $(G,P,T)$, and hence has zero Kolmogorov-Sinai entropy. 
On the other hand, the topological 
entropy of $(X_\eta, S)$ is positive, equal to $\prod_{\ell\geq 1}\big(1-1/N(\bb_\ell)\big)$ (Theorem B in \cite{ADK}). Next, we outline the characterization of $X_\eta$ in terms of admissible sequences from \cite{ADK}. 
For $A\subseteq\OK$
and a nonzero ideal $\bb$ write
\begin{equation}\label{eq:def-classcount}
|A/\bb|:=\big|\{a+\bb:a\in A\}\big|\in\{0,1,\dots,N(\bb)\}
\end{equation}
for the number of residue classes modulo $\bb$ met by $A$. 
\begin{definition}\label{def:adm}
A configuration $x\in\{0,1\}^{\OK}$ is \emph{$\bB$-admissible} if
$|\operatorname{supp}x/\bb_\ell|<N(\bb_\ell)$ for every $\ell\ge1$. That is, the
support of $x$ omits at least one residue class modulo each $\bb_\ell$. The set $X_{\bB}$ of
all $\bB$-admissible configurations is a closed, shift-invariant subset of
$\{0,1\}^{\OK}$, which we refer to as the \emph{$\bB$-free subshift}.
\end{definition}
 The
subshift $X_{\bB}$ is \emph{hereditary}: if $x\in X_{\bB}$ and $x'\le x$ coordinatewise,
then $x'\in X_{\bB}$. 
Theorem C in \cite{ADK} states that $X_\eta=X_{\bB}$, and justifies why we treat the subshifts  $(X_\eta,S)$ and  $(X_{\bB},S)$ interchangeably, as in the classical squarefree case.


We stress once more that the framework just described subsumes all the earlier settings:
the classical case is $K=\Q$, and the $k$-free, lattice, and number-field cases
arise as special or conjugate instances, as discussed in \cite{ADK}.

\subsection{The main theorem: intrinsic ergodicity for the $\bB$-free subshift}\label{ss:mainthm}
In every one of the multidimensional and number-field settings above, the
structural properties in Sarnak's program above are now known, but intrinsic ergodicity is not. The
uniqueness of the measure of maximal entropy has been established only in
dimension one~\cite{DKS,KLW,Pe}. Our goal is to provide intrinsic ergodicity for the general $\bB$-free subshift described above. Above dimension one, this is a new result, see \cite[Remark 4.18]{Ar2}. Intrinsic ergodicity is not a formality (hereditary subshifts need not be intrinsically ergodic in general
\cite{KLW,Kw}) so the property must be earned from the arithmetic structure. The purpose of this work is to show intrinsic ergodicity. Before stating our main theorem precisely, we introduce the shift-invariant measure $\kappa$ on $X_\bB$.

Let us define the \emph{coding map} $\varphi\colon G\to\{0,1\}^{\OK}$ by
\begin{equation}\label{eq:def-phi}
\varphi(g)(a)=1\iff g_\ell+a\neq0\ \text{in}\ \OK/\bb_\ell\ \text{ for every }\ell\ge1 .
\end{equation}
Note that $\varphi(0)=\eta$ is the configuration in $\{0,1\}^{\OK}$ that forbids the class $0$ modulo  $\bb_\ell$ for every $\ell\geq1$. When dealing with the coding map and related concepts, we refer to $g\in G$ as a \emph{phase}, since $\varphi(g)$ shifts each periodic `layer' $\OK/\bb_\ell$ independently by $g_\ell$. Clearly, $\varphi(g)\in X_{\bB}$ for all phases $g\in G$.
We denote by  
\begin{equation}\label{eq:def-Ag}
A_g:=\operatorname{supp}\varphi(g)=\{a\in\OK:a\not\equiv-g_\ell\ (\mathrm{mod}\ \bb_\ell)\ \text{for every }\ell\ge1\}
\end{equation}
the support of $\varphi(g)$, i.e. the set of \emph{allowed positions} for the phase $g\in G$.
We also point out that $\varphi$ is Borel\footnote{To see that $\varphi$ is Borel,  consider the closed set
$C:=\{g\in G:g_\ell\neq0\ \text{for every }\ell\ge1\}$. Then $\varphi(g)(a)=\mathbf 1_C(T_ag)$ (cf.\ Remark~2.7 of~\cite{ADK}), so
every coordinate of $\varphi$ is Borel.}. 
The push-forward
$\nu_\eta:=\varphi_*P$ is the aforementioned \emph{Mirsky measure} on $X_{\bB}$.
Let $\beta$ be the uniform Bernoulli measure on
$\{0,1\}^{\OK}$ (independent fair coin flips at each integer in $\OK$) and define 
\begin{equation}\label{eq:def-Phi}
\Phi\colon G\times\{0,1\}^{\OK}\to X_{\bB},\qquad\Phi(g,\omega)=\varphi(g)\cdot\omega,
\end{equation} where ``$\cdot$'' denotes the coordinatewise product. Finally, set
\begin{equation}\label{eq:def-kappa}
\kappa:=\Phi_*(P\otimes\beta).
\end{equation}

In the statement below, $\htop$ denotes the topological entropy of the
$\OK$-action and $h_\nu(S)$ the measure-theoretic (Kolmogorov--Sinai) entropy of an
invariant measure $\nu$; both notions, and the associated relative entropy used
throughout the paper, are recalled in Section~\ref{ss:entropy}.

\begin{theorem}[The $\bB$-free subshift is intrinsically ergodic]\label{thm:main}
Let $K$ be a number field with ring of integers $\OK$, and let $\bB$ be an
Erd\H{o}s family of ideals in $\OK$. Then the $\bB$-free subshift
$(X_{\bB},S)$ is intrinsically ergodic: it has a unique measure of
maximal entropy, namely $\kappa$. Equivalently, conditioned on a Haar-random
phase $g\in G$, the unique  measure of maximal entropy places independent fair coins on the
allowed positions $A_g$ and zeros elsewhere. In particular
\begin{equation}\label{eq:entropy-value}
h_\kappa(S)=\htop(X_{\bB})=\prod_{\ell
\geq 1}\Big(1-\frac{1}{N(\bb_\ell)}\Big),
\end{equation}
and $\kappa$ is mutually singular with the zero-entropy Mirsky measure
$\nu_\eta$.
\end{theorem}

Theorem~\ref{thm:main} exhibits the measure of maximal entropy explicitly, and
the description makes its contrast with the Mirsky measure clear. Both
measures are carried by $Y$ and both project to the Haar measure $P$ under the
phase map: they are skew products over the same rotation, and they
differ only in the randomness of their fibres. Conditionally on a $P$-typical
phase $g$, the measure $\kappa$ tosses an independent fair coin at every
allowed position in $A_g$ (and writes $0$ elsewhere), extracting one full bit
of entropy per allowed site; this is the source of the value $h_\kappa(S)=D$. The
Mirsky measure corresponds to the degenerate case in which all coins come up heads:
conditionally on the phase, $\nu_\eta$ is the Dirac mass at $\varphi(g)$, the
configuration occupying every allowed position, and no entropy is produced.
This dichotomy also makes the mutual singularity explicit: $\kappa$-almost
surely the density of occupied sites among the allowed ones equals $\tfrac12$,
whereas $\nu_\eta$-almost surely it equals $1$. Since $\eta$ itself is generic
for $\nu_\eta$, the generic $\bB$-free configuration is, from the point of view
of maximal entropy, atypical: it is the all-heads outcome, the single most
orderly configuration compatible with its phase.

Because~\cite{ADK} realises the $k$-free and $\bB$-free lattice-point systems and
the $k$-free number-field system as special cases or topological conjugates of
the systems covered by Theorem~\ref{thm:main}, and intrinsic ergodicity is a
conjugacy invariant, we obtain the following.

\begin{corollary}\label{cor:transfer}
The $k$-free lattice-point subshifts of Pleasants--Huck~\cite{PH} and
Huck--Baake~\cite{HB}, the $\bB$-free lattice-point subshifts of
Baake--Huck~\cite{BH}, and the $k$-free number-field subshifts of
Cellarosi--Vinogradov~\cite{CV} are all intrinsically ergodic.
\end{corollary}
Theorem \ref{thm:main} also yields an equidistribution statement for uniformly random admissible patterns; see Corollary \ref{cor:equid}.

\subsection{Strategy of proof}\label{ss:strategy}

Although the maximal measure is of positive entropy and is \emph{not} the rotation
measure $\nu_\eta$; nevertheless the rotation is the organising factor of the
problem. Our argument proceeds in three reductions (Section~\ref{sec:reduction})
and one rigidity step, which we prove twice.

The starting point, taken from~\cite{ADK}, is that \emph{any} measure of maximal entropy
is concentrated on the \emph{saturated set} $Y$, defined as 
\begin{equation}\label{eq:def-Y}
Y:=\big\{x\in X_{\bB}:\ |\operatorname{supp}x/\bb_\ell|=N(\bb_\ell)-1\ \text{for every }\ell\ge1\big\}\subseteq X_{\bB},
\end{equation}
i.e. the set of all admissible configurations whose support omits \emph{exactly one} class modulo
each $\bb_\ell$. For $x\in Y$ and each $\ell\geq 1$ there is a unique
$r_\ell(x)\in\OK/\bb_\ell$ with $\operatorname{supp}x\cap(r_\ell(x)+\bb_\ell)=\varnothing$. The \emph{phase map} $\theta\colon Y\to G$ is defined as
\begin{equation}\label{eq:def-theta}
\theta(x):=(-\,r_\ell(x))_{\ell\geq1} .
\end{equation}
The sign is chosen so that $\theta(\varphi(g))=g$ for all $g\in G$, the class
omitted by $A_g$ modulo $\bb_\ell$ being $-g_\ell$. Both $\varphi$ and $\theta$
are Borel, and equivariant:
\begin{equation}\label{eq:equivariance}
S_a\circ\varphi=\varphi\circ T_a,\qquad T_a\circ\theta=\theta\circ S_a\qquad(a\in\OK).
\end{equation}
By unique ergodicity of $(G,T)$, every
invariant measure on $Y$ pushes forward to $P$; and by an Abramov--Rokhlin
entropy addition formula for $\OK$($\cong\Z^d$)-actions, the entropy of such a
measure equals its relative entropy over the rotation factor, the base
contributing nothing because it has zero entropy. The problem therefore reduces
to a single relative statement (Proposition~\ref{prop:rigidity}): among measures
projecting to $P$, the relative entropy is bounded above by the density
$\prod_{\ell\geq 1}(1-1/N(\bb_\ell))$ of the allowed positions, with equality attained
\emph{only} by $\kappa$.

We prove this rigidity statement in two separate ways. The first
(Section~\ref{sec:relent}) is a  measure-theoretic argument:
expressing the relative entropy as the conditional entropy of a single site given
its lexicographic past and the phase, equality forces each allowed coordinate to
be a fair coin independent of its past and of the phase, which determines the
conditional law and hence the measure. The second (Section~\ref{sec:induce})
follows Peckner's induce-and-split scheme~\cite{Pe}, in which the skew product
over the rotation is split into a direct product of the base with a full shift;
since $\OK$ is multidimensional and has no first-return map, the inducing step is
replaced by an exact tiling of $\OK$ by cubes~\cite{DHZ, OW}, and uniqueness of the
product maximal measure transports back to $X_{\bB}$. We stress that the maximal
equicontinuous factor of the proximal system $X_{\bB}$ is trivial and the
rotation is only a \emph{measurable} factor of $X_{\bB}$. This is why our
analysis is measure-theoretic.

The rest of the paper is organised as follows. 
Section~\ref{sec:prelim}  reviews the entropy theory for $\OK$-actions and  collects the needed facts from~\cite{ADK}. Section~\ref{sec:reduction} reduces Theorem~\ref{thm:main} to the rigidity
Proposition~\ref{prop:rigidity}. Sections~\ref{sec:relent}
and~\ref{sec:induce} give the two proofs of that proposition.
Section~\ref{sec:proof} assembles the argument, proves Theorem~\ref{thm:main}
and Corollary~\ref{cor:transfer}, and discusses other consequences of intrinsic ergodicity.

\section{Preliminaries}\label{sec:prelim}
\subsection{Entropy and relative entropy for $\OK$-actions}\label{ss:entropy}
The group $\OK\cong\Z^d$ is countable amenable and we fix a tempered F\o lner sequence
$(F_n)$ in $\OK$ (for definiteness, the Minkowski boxes of~\cite{ADK}). 
%
For a continuous $\OK$-action $S=(S_a)_{a\in\OK}$ on a compact metrisable space $X$ we write
$\htop(X,S)$ for its topological entropy; see, e.g.,
\cite[Chapter~9]{KL} for the entropy theory of actions of countable amenable
groups. Given a Borel probability measure
$\nu$ on $X$, the \emph{Shannon entropy} of a finite partition
$\mathcal Q=\{Q_1,\dots,Q_k\}$ of $X$ into Borel sets is
\begin{equation}\label{eq:def-shannon}
H_\nu(\mathcal Q):=-\sum_{i=1}^{k}\nu(Q_i)\log_2\nu(Q_i),
\end{equation}
with the convention $0\log_20:=0$. For an $S$-invariant Borel
probability measure $\nu$ and a finite partition $\mathcal Q$, set
\begin{equation}\label{eq:def-hSQ}
    h_\nu(S,\mathcal Q)=\lim_{n\to\infty}\frac{1}{|F_n|}H_\nu\big(\bigvee_{a\in F_n}S_a^{-1}\mathcal Q\big),
    \end{equation}
    where the limit exists by amenability~\cite{OW}. The
Kolmogorov--Sinai entropy is  
\begin{equation}\label{eq:def-KS}
h_\nu(S)=\sup_{\mathcal Q}h_\nu(S,\mathcal Q),\end{equation}
where the supremum is over all finite Borel partitions of $X$.
The variational principle for actions of countable
amenable groups~\cite{MOP} states that
\begin{equation}\label{variational-principle}
\htop(X, S)=\sup_\nu h_\nu(S),
\end{equation}where the supremum is over all $S$-invariant Borel probability measures $\nu$ on $X$. An invariant measure
attaining the supremum in \eqref{variational-principle} is a \emph{measure of maximal entropy}.
 
Given an $S$-invariant sub-$\sigma$-algebra $\mathscr F$ of the Borel $\sigma$-algebra $\mathscr B_X$
and a finite partition $\mathcal Q=\{Q_1,\dots,Q_k\}$ of $X$,
the \emph{conditional Shannon entropy} 
is
\begin{equation}\label{eq:def-condent}
H_\nu(\mathcal Q\mid\mathscr F)
:=\int_X\left(\sum_{i=1}^k-\,\mathbb E_\nu[\mathbf 1_{Q_i}\mid\mathscr F]
\log_2\mathbb E_\nu[\mathbf 1_{Q_i}\mid\mathscr F]\right)\,d\nu.
\end{equation}
Analogously to the unconditional setting, we define
\begin{equation}\label{eq:def-relent1}
h_\nu(S,\mathcal Q\mid\mathscr F)=\lim_{n\to\infty}\frac1{|F_n|}H_\nu\Big(\textstyle\bigvee_{a\in F_n}S_a^{-1}\mathcal Q\,\Big|\,\mathscr F\Big)
\end{equation}
and the \emph{relative entropy} is
\begin{equation}\label{eq:def-relent2}
h_\nu(S\mid\mathscr F)=\sup_{\mathcal Q}h_\nu(S,\mathcal Q\mid\mathscr F),
\end{equation}
where the supremum is again over all finite Borel partitions of $X$. We
freely use the following standard properties of conditional entropy
\ref{itm:E1}--\ref{itm:E5} (see, e.g., \cite[Chapter~1]{Do}), together with
the dynamical property \ref{itm:E6}:
\begin{enumerate}[label=\textup{(E\arabic*)},ref={(E\arabic*)},leftmargin=3em]
  \item\label{itm:E1} For a partition $\mathcal Q=\{Q_1,\dots,Q_k\}$, we have $H_\nu(\mathcal Q\mid\mathscr F)\le H_\nu(\mathcal Q)\le\log_2k$.
  \item\label{itm:E2} For $\sigma$-algebras $\mathscr F\subseteq\mathscr F'$, we have $H_\nu(\mathcal Q\mid\mathscr F')\le
  H_\nu(\mathcal Q\mid\mathscr F)$.
  \item\label{itm:E3} The chain rule:
  $H_\nu(\mathcal Q\vee\mathcal R\mid\mathscr F)
  =H_\nu(\mathcal Q\mid\mathscr F)+H_\nu(\mathcal R\mid\sigma(\mathcal Q)\vee\mathscr F)$.
  \item\label{itm:E4} Conditional subadditivity:
  $H_\nu(\mathcal Q\vee\mathcal R\mid\mathscr F)\le
  H_\nu(\mathcal Q\mid\mathscr F)+H_\nu(\mathcal R\mid\mathscr F)$, with equality
  if and only if $\mathcal Q$ and $\mathcal R$ are conditionally independent
  given $\mathscr F$.
  \item\label{itm:E5} Martingale continuity: 
  $H_\nu(\mathcal Q\mid\mathscr F_m)\downarrow H_\nu(\mathcal Q\mid\mathscr F)$
  whenever $\mathscr F_m\uparrow\mathscr F$ is an increasing sequence of
  sub-$\sigma$-algebras generating $\mathscr F$.
  \item\label{itm:E6} The Kolmogorov--Sinai theorem and its relative version
  for amenable group actions~\cite{OW,WZ}: if $\mathcal P$ is a
  \emph{generating} partition for the action $S$ on $X$ (i.e. the $\sigma$-algebra generated by  $\bigvee_{a\in\OK}S_a^{-1}\mathcal P$ is the Borel $\sigma$-algebra $\mathscr B_{X}$), then
  $h_\nu(S)=h_\nu(S,\mathcal P)$ and, for every $S$-invariant
  sub-$\sigma$-algebra $\mathscr F$,
  $h_\nu(S\mid\mathscr F)=h_\nu(S,\mathcal P\mid\mathscr F)$.
\end{enumerate}
For $X=X_{\bB}$ the
time-$0$ partition $\mathcal P=\{[x_0=0],[x_0=1]\}$ is generating,
so by \ref{itm:E6} both suprema in \eqref{eq:def-KS} and
\eqref{eq:def-relent2} are attained at $\mathcal P$. If
$\pi\colon(X,\nu,S)\to(Z,\pi_*\nu,R)$ is a measurable factor map of
$\OK$-systems (with $S=(S_a)_{a\in\OK}$, $R=(R_a)_{a\in\OK}$, and $\pi\circ S_a=R_a\circ\pi$ $\nu$-almost everywhere for every $a\in\OK$) and $\mathscr F=\pi^{-1}(\mathscr B_Z)$ (an $(S_a)$-invariant
sub-$\sigma$-algebra), then the Abramov--Rokhlin addition formula for amenable
group actions~\cite{WZ} gives
\begin{equation}\label{eq:AR}
h_\nu(S)=h_{\pi_*\nu}(R)+h_\nu(S\mid\mathscr F).
\end{equation}

\subsection{Facts from Ara\'ujo--Dymek--Ku\l aga-Przymus~\texorpdfstring{\cite{ADK}}{[ADK]}}\label{ss:facts}
We collect here various facts, either established
in~\cite{ADK} or standard, which we will need. Recall the $\OK$-actions $S$ and $T$ from \eqref{eq:S-and-T}.
\begin{enumerate}[label=\textup{(F\arabic*)},ref={(F\arabic*)},leftmargin=3em]
  \item\label{itm:F1} One has $\htop(X_{\bB},S)=D$, where $D:=\prod_{\bb\in\bB}(1-1/N(\bb))>0$
  (the product converges and is positive by Definition~\ref{def:erdos}).
  \item\label{itm:F2} Every measure of maximal entropy of $(X_{\bB},S)$ is concentrated on $Y$.
  \item\label{itm:F3} The system $(G,P,T)$ is uniquely ergodic, with $h_P(T)=0$.
  \item\label{itm:F4} The maps $\varphi\colon G\to X_{\bB}$ and $\theta\colon Y\to G$ are Borel and
  equivariant (see Sections~\ref{ss:mainthm} and~\ref{ss:strategy} above), and $\nu_\eta=\varphi_*P$.
  \item\label{itm:F5} The space $X_{\bB}$ is a subshift over the finite alphabet $\{0,1\}$, hence admits at
  least one measure of maximal entropy; and in the ergodic decomposition of a
  measure of maximal entropy almost every ergodic component is again of maximal
  entropy.\footnote{Precisely: every $S$-invariant Borel probability measure
  $\nu$ on $X_{\bB}$ admits a unique \emph{ergodic decomposition}
  $\nu=\int_{\mathcal E}m\,d\lambda_\nu(m)$, where $\mathcal E$ is the set of
  ergodic $S$-invariant Borel probability measures on $X_{\bB}$ and
  $\lambda_\nu$ is a Borel probability measure on $\mathcal E$. Moreover, for
  actions of countable amenable groups the entropy map satisfies the integral
  formula $h_\nu(S)=\int_{\mathcal E}h_m(S)\,d\lambda_\nu(m)$; see,
  e.g.,~\cite{OW}. If $h_\nu(S)=\htop(X_{\bB})$ then, since
  $h_m(S)\le\htop(X_{\bB})$ for every $m\in\mathcal E$ by the variational
  principle, the integral formula forces $h_m(S)=\htop(X_{\bB})$ for
  $\lambda_\nu$-almost every $m\in\mathcal E$. This is the meaning of ``almost
  every ergodic component is again of maximal entropy''.}
\end{enumerate}

\section{Reduction to a rigidity statement}\label{sec:reduction}

Recall the measure $\kappa$ defined in \eqref{eq:def-kappa}, and set
$\Ftheta:=\theta^{-1}(\mathscr B_G)$ (an $S$-invariant $\sigma$-algebra, by
\ref{itm:F4}). The next three propositions reduce Theorem~\ref{thm:main} to a single
rigidity statement.

\begin{proposition}[Projection to the Haar factor; cf.~\cite{ADK}]\label{prop:haar}
Every $(S_a)$-invariant $\nu$ with $\nu(Y)=1$ satisfies $\theta_*\nu=P$.
\end{proposition}
\begin{proof}
Since $\nu(Y)=1$, the push-forward $\theta_*\nu$ is defined; by \ref{itm:F4} it is a
$(T_a)$-invariant measure on $G$, and by unique ergodicity \ref{itm:F3} it equals $P$.
This observation is made in~\cite{ADK}, in the course of the proofs underlying
their isomorphism theorem for the Mirsky measure; we have restated it for ease
of reference.
\end{proof}

\begin{proposition}[Skew-product presentation of $\kappa$]\label{prop:skew}
The map $\Phi$ defined in \eqref{eq:def-Phi} is Borel, and it
intertwines the diagonal $\OK$-action $(T_a\times S_a)_{a\in\OK}$ on its domain
with the shift action on $X_{\bB}$:
\begin{equation}\label{eq:Phi-equivariance}
\Phi\circ(T_a\times S_a)=S_a\circ\Phi\qquad\text{for every }a\in\OK.
\end{equation}
Moreover, the measure $\kappa$ is $S$-invariant, ergodic, and satisfies
$\kappa(Y)=1$ and $\theta_*\kappa=P$.
\end{proposition}
\begin{proof}
The map $\Phi$ is Borel because each coordinate
$\Phi(g,\omega)(a)=\varphi(g)(a)\,\omega(a)$ is the product of a Borel and a
continuous function of $(g,\omega)$.
Equivariance of $\Phi$ follows from $S_a\varphi=\varphi T_a$ together with
$S_a(x\cdot\omega)=(S_ax)\cdot(S_a\omega)$; since $P\otimes\beta$ is invariant
under the diagonal action, $\kappa$ is $S$-invariant.

Let us now show ergodicity. The Bernoulli system $(\{0,1\}^{\OK},\beta,S)$ is
mixing, hence weakly mixing, and the product of an ergodic system with a weakly
mixing system is ergodic; so $P\otimes\beta$ is ergodic for the diagonal action
and its factor $\kappa$ is ergodic.

Let us now show that $\kappa(Y)=1$ and $\theta_*\kappa=P$. By construction
$\operatorname{supp}\Phi(g,\omega)\subseteq A_g$, which avoids the class
$-g_\ell$ modulo $\bb_\ell$ for every $\ell$. It remains to check that, almost
surely, every \emph{other} class is met.
Fix $\ell$ and $r\in\OK/\bb_\ell$ with $r\neq-g_\ell$, and pick
$a_0\in\OK$ with $\pi_\ell(a_0)=r$. For $a\in a_0+\bb_\ell$ the condition $a\in
A_g$ reads $a\not\equiv-g_m\ (\mathrm{mod}\ \bb_m)$ for all $m\neq\ell$ (the
$\ell$-th constraint holds automatically) and thus depends only on
$g^{(\ell)}:=(g_m)_{m\neq\ell}$. By pairwise coprimality the image of the ideal
$\bb_\ell$ in any finite product $\prod_{m\in M}\OK/\bb_m$ (with $\ell\notin M$)
is all of it, so the diagonal image of $\bb_\ell$ is dense in
$G^{(\ell)}:=\prod_{m\neq\ell}\OK/\bb_m$ and the $\bb_\ell$-rotation on
$G^{(\ell)}$ with its Haar measure is uniquely ergodic. The pointwise ergodic
theorem along F\o lner sets of $\bb_\ell\cong\Z^d$ then gives, for $P$-a.e.\
$g$, that $\{a\in a_0+\bb_\ell:\ a\in A_g\}$ has density
$\prod_{m\neq\ell}(1-1/N(\bb_m))>0$ in $a_0+\bb_\ell$; in particular this set is
infinite. Since $\omega$ is $\beta$-distributed independently of $g$, and fair
coins indexed by an infinite set are almost surely not all $0$, we get that almost surely
$\omega_a=1$ for some allowed $a$ in the class $r$. Then, intersecting over the
countably many pairs $(\ell,r)$ preserves full measure.
Hence for $P\otimes\beta$-a.e.\ $(g,\omega)$ the support of $\Phi(g,\omega)$
meets exactly the $N(\bb_\ell)-1$ classes distinct from $-g_\ell$ modulo each
$\bb_\ell$, i.e.\ $\Phi(g,\omega)\in Y$ and $\theta(\Phi(g,\omega))=g$. Thus
$\kappa(Y)=1$ and $\theta_*\kappa=P$, as claimed
\end{proof}

\begin{proposition}[Abramov--Rokhlin decomposition]\label{prop:abramov}
For any ergodic $S$-invariant measure $\nu$ with $\nu(Y)=1$ and $\theta_*\nu=P$,
$h_\nu(S)=h_\nu(S\mid\Ftheta)$.
\end{proposition}
\begin{proof}
Simply apply the Abramov--Rokhlin addition formula \eqref{eq:AR} to the measurable factor
$\theta\colon(X_{\bB},\nu,S)\to(G,P,T)$ and observe that the base term vanishes since
$h_P(T)=0$ by \ref{itm:F3}.
\end{proof}

\begin{proposition}[Rigidity]\label{prop:rigidity}
For every $S$-invariant measure $\nu$ with $\nu(Y)=1$ and $\theta_*\nu=P$ one has
$h_\nu(S\mid\Ftheta)\le D$, and equality holds if and only if $\nu=\kappa$.
\end{proposition}

Proposition~\ref{prop:rigidity} is the key technical contribution of this work. We give two separate proofs of it: in Section~\ref{sec:relent} by a relative-entropy
computation, and in Section~\ref{sec:induce} by an induce-and-split argument.
Either proof, combined with Propositions~\ref{prop:haar}--\ref{prop:abramov} and
the facts \ref{itm:F1}--\ref{itm:F5}, yields Theorem~\ref{thm:main}; the assembly is carried out
in Section~\ref{sec:proof}.

\section{First proof of the rigidity: relative entropy}\label{sec:relent}

Throughout this section, $\nu$ denotes an $S$-invariant measure with $\nu(Y)=1$ and
$\theta_*\nu=P$, as in Proposition~\ref{prop:rigidity}. 
Fix
the group isomorphism $\iota\colon\Z^d\to\OK$ from~\cite{ADK}
and let $\prec$ be the pullback of the lexicographic order on $\Z^d$ (a
translation-invariant total order). For $W\subseteq\OK$ we write
$(x_b)_{b\in W}\vee\Ftheta$ for the $\sigma$-algebra generated by the coordinates
$x_b$, $b\in W$, together with $\Ftheta$ (defined in Section \ref{sec:reduction}). We also set
$\mathscr P^-:=\sigma\big(\bigvee_{b\prec0}S_b^{-1}\mathcal P\big)$ (the $\sigma$-algebra generated by the lexicographic past) and
$\mathscr G:=\mathscr P^-\vee\Ftheta$. Finally, for $a\in\OK$ we write
$\{a\in A_g\}$ for the event
$\{x\in Y:\ a\in A_{\theta(x)}\}=\theta^{-1}\big(\{g\in G:\ a\in A_g\}\big)$,
which belongs to $\Ftheta$.

\begin{lemma}[Relative entropy as a single-site conditional entropy]\label{lem:site}
One has $h_\nu(S\mid\Ftheta)=H_\nu\big(x_0\mid\mathscr G\big)$.
\end{lemma}
\begin{proof}
By \ref{itm:E6} it suffices
to compute $h_\nu(S,\mathcal P\mid\Ftheta)$, where $\mathcal P$ is the time-0 partition. Enumerate
$F_n=\{a_1\prec\dots\prec a_{k_n}\}$ in increasing lexicographic order. The
chain rule \ref{itm:E3} for conditional entropy, the $S$-invariance of $\nu$, and the
$S$-invariance of $\Ftheta$ give
\begin{align}
H_\nu\Big(\bigvee_{a\in F_n}S_a^{-1}\mathcal P\,\Big|\,\Ftheta\Big)
&=\sum_{i=1}^{k_n}H_\nu\big(x_0\,\big|\,(x_b)_{b\in W_{n,i}}\vee\Ftheta\big),\label{eq:site-chainrule} 
\end{align}
where $W_{n,i}:=\{a_j-a_i:\,j<i\}\subseteq\{b:b\prec0\}$.
Using \ref{itm:E2} we see that each summand in \eqref{eq:site-chainrule} is at least
$H_\nu(x_0\mid\mathscr G)$. Dividing by $|F_n|$, we get
$h_\nu(S,\mathcal P\mid\Ftheta)\ge H_\nu(x_0\mid\mathscr G)$. For the reverse
bound fix $m\ge1$ and set $B_m:=\iota([-m,m]^d)\cap\{b:b\prec0\}$. Note that if $a_i$ lies
in the $m$-interior $\{a\in F_n:\ a+\iota([-m,m]^d)\subseteq F_n\}$, then
$W_{n,i}\supseteq B_m$, so  by \ref{itm:E2} the $i$-th summand is at most
$H_\nu(x_0\mid(x_b)_{b\in B_m}\vee\Ftheta)$. Moreover, by  \ref{itm:E1} every summand is at most $\log_22=1$, and by the F\o lner property the number of non-interior indices is $o(|F_n|)$.
Hence $h_\nu(S,\mathcal P\mid\Ftheta)\le H_\nu(x_0\mid(x_b)_{b\in B_m}\vee\Ftheta)$
for every $m$, and as $m\to\infty$ the right-hand side decreases to
$H_\nu(x_0\mid\mathscr G)$ by \ref{itm:E5} along the increasing sequence of
$\sigma$-algebras $(x_b)_{b\in B_m}\vee\Ftheta\uparrow\mathscr G$.
\end{proof}

\begin{lemma}[Forbidden sites; density of allowed sites]\label{lem:forbidden}
For $\nu$-a.e.\ $x$ and every $a\notin A_{\theta(x)}$ one has $x_a=0$.
Consequently
$\nu(\{0\in A_g\})=P(\{0\in A_g\})
=D$.
\end{lemma}
\begin{proof}
Since $\nu(Y)=1$, we may write $g=\theta(x)$ for $\nu$-a.e.\ $x$. If
$a\notin A_g$ then $a\equiv-g_\ell\ (\mathrm{mod}\ \bb_\ell)$
for some $\ell$, so $a\notin\operatorname{supp}x$, i.e.\ $x_a=0$
(by definition of the phase map, $\operatorname{supp}x$ omits precisely
the class $r_\ell(x)=-g_\ell$ modulo $\bb_\ell$, and $a\notin A_g$ means
$a\equiv-g_\ell$ for some $\ell$). Moreover
$\{0\in A_g\}=\bigcap_\ell\{g_\ell\ne0\}$ is an intersection of independent
uniform events and has $P$-measure $\prod_{\bb\in\bB}(1-1/N(\bb))=D$.  The event lies
in $\Ftheta$, and $\theta_*\nu=P$ gives the last equality.
\end{proof}

\begin{proof}[First proof of Proposition~\ref{prop:rigidity}]
Let us start by proving the bound $h_\nu(S\mid\Ftheta)\le D$. Let
$p(x):=\nu(x_0=1\mid\mathscr G)(x)$ and let
$h(x):=-p(x)\log_2p(x)-(1-p(x))\log_2(1-p(x))$ be the corresponding binary
entropy. By Lemma~\ref{lem:forbidden}, $h=0$ on the
$\mathscr G$-measurable event $\{0\notin A_g\}$ and by \ref{itm:E1} $h\le\log_22=1$ on $\{0\in A_g\}$, so by Lemma~\ref{lem:site} we have
\begin{equation}\label{eq:bound-D}
h_\nu(S\mid\Ftheta)=\mathbb E_\nu[h]=\mathbb E_\nu[h\,\mathbf 1_{\{0\in A_g\}}]
\le\nu(\{0\in A_g\})=D.
\end{equation}
Let us now show assume that the equality $h_\nu(S\mid\Ftheta)=D$ holds. Then \eqref{eq:bound-D} is an equality, so $h=1$
$\nu$-a.e.\ on $\{0\in A_g\}$. Therefore the function $\nu(x_0=1\mid\mathscr G): X_{\bB}\to[0,1]$ equals $\tfrac12$ a.e.\
on the set $\{0\in A_g\}$\footnote{We are using the fact that 
$p\mapsto-p\log_2p-(1-p)\log_2(1-p)$ attains
its maximum value $1$ only at $p=\frac{1}{2}$.}. We now claim 
that, for every $a$,
\begin{equation}\label{eq:fair-coin}
\nu\big(x_a=1\mid(x_b)_{b\prec a}\vee\Ftheta\big)=\frac{1}{2}
\quad\text{$\nu$-a.e.\ on }\{a\in A_g\}.
\end{equation}
The identity \eqref{eq:fair-coin} asserts that any version\footnote{Recall that, for a sub-$\sigma$-algebra
$\mathscr A\subseteq\mathscr B_{X_{\bB}}$, a \emph{version} of
$\nu(x_a=1\mid\mathscr A)$ is any $\mathscr A$-measurable function
$f\colon X_{\bB}\to[0,1]$ satisfying $\int_Af\,d\nu=\nu(\{x_a=1\}\cap A)$ for
every $A\in\mathscr A$. Such a function exists by the Radon--Nikodym theorem
and is unique up to $\nu$-null sets, so the assertion in
\eqref{eq:fair-coin} does not depend on the choice of version.} of the conditional
probability on its left-hand side is equal to $\frac{1}{2}$ at $\nu$-almost
every point of the event $\{a\in A_g\}$. 
To derive \eqref{eq:fair-coin}, fix $a\in\OK$. We use the following
change-of-variables identity: if $\phi\colon X_{\bB}\to X_{\bB}$ is
$\nu$-preserving and $\mathscr H\subseteq\mathscr B_{X_{\bB}}$ is a
sub-$\sigma$-algebra, then, for every bounded Borel function $f$,
\begin{equation}\label{eq:cov}
\mathbb E_\nu\big[f\circ\phi\mid\phi^{-1}\mathscr H\big]
=\mathbb E_\nu\big[f\mid\mathscr H\big]\circ\phi\qquad\text{$\nu$-a.e.}
\end{equation}
We apply
\eqref{eq:cov} with $\phi=S_a$, $f=\mathbf 1_{\{x_0=1\}}$ and
$\mathscr H=\mathscr G$. First, $f\circ S_a=\mathbf 1_{\{x_a=1\}}$, since
$(S_ax)_0=x_a$. Second, $S_a^{-1}\mathscr G=(x_b)_{b\prec a}\vee\Ftheta$: in fact
the $\sigma$-algebra $S_a^{-1}\mathscr P^-$ is generated by the coordinates
$(x_{a+b})_{b\prec0}=(x_c)_{c\prec a}$ because $\prec$ is
translation-invariant, and $S_a^{-1}\Ftheta=\Ftheta$ because
$\theta\circ S_a=T_a\circ\theta$ and $T_a$ preserves $\mathscr B_G$. Therefore
\eqref{eq:cov} yields
$\nu\big(x_a=1\mid(x_b)_{b\prec a}\vee\Ftheta\big)=p\circ S_a$ $\nu$-a.e.
Finally, $S_a^{-1}\{0\in A_g\}
=\{x:0\in A_{T_a\theta(x)}\}=\{x:a\in A_{\theta(x)}\}=\{a\in A_g\}$ by
$A_{T_ag}=A_g-a$; since $p=\frac{1}{2}$ $\nu$-a.e.\ on $\{0\in A_g\}$ and
$S_a^{-1}$ maps $\nu$-null sets to $\nu$-null sets, we conclude that
$p\circ S_a=\frac{1}{2}$ $\nu$-a.e.\ on $\{a\in A_g\}$, which is precisely
\eqref{eq:fair-coin}.
Next, we claim that the finite-dimensional conditional laws of $\nu$ given $\mathscr F_\theta$ agree with those of $\kappa$. That is, for finite $F=\{a_1\prec\dots\prec a_m\}$
and $w\in\{0,1\}^F$, we claim that 
\begin{equation}\label{eq:findim}
\nu\big(x|_F=w\mid\Ftheta\big)=
\begin{cases}
\big(\frac{1}{2}\big)^{\,|F\cap A_g|}, & w_a=0\ \forall a\in F\setminus A_g,\\
0, & \text{otherwise.}
\end{cases}
\end{equation}
Indeed, for each $i$ the $\sigma$-algebra
$(x_{a_j})_{j<i}\vee\Ftheta$ is coarser than $(x_b)_{b\prec a_i}\vee\Ftheta$, so
the tower property applied to \eqref{eq:fair-coin} gives, on the $\Ftheta$-measurable event
$\{a_i\in A_g\}$,
\begin{equation}\label{eq:tower}
\nu\big(x_{a_i}=1\mid(x_{a_j})_{j<i}\vee\Ftheta\big)
=\mathbb E_\nu\Big[\nu\big(x_{a_i}=1\mid(x_b)_{b\prec a_i}\vee\Ftheta\big)\,\Big|\,
(x_{a_j})_{j<i}\vee\Ftheta\Big]=\frac{1}{2},
\end{equation}
while on $\{a_i\notin A_g\}$ one has $x_{a_i}=0$ a.s.\ by
Lemma~\ref{lem:forbidden}. Iterating the chain rule
$\nu\big(\{x|_{F'}=w|_{F'}\}\cap\{x_{a_i}=w_i\}\mid\Ftheta\big)
=\mathbb E_\nu\big[\mathbf 1_{\{x|_{F'}=w|_{F'}\}}\,
\nu(x_{a_i}=w_i\mid(x_{a_j})_{j<i}\vee\Ftheta)\mid\Ftheta\big]$ with
$F'=\{a_1,\dots,a_{i-1}\}$, from $i=m$ downwards, yields the claim.
This is exactly the conditional law of $\kappa$ given $\Ftheta$.
Both $\nu$ and $\kappa$ are carried by $Y$ and project to $P$ under
$\theta$, and their conditional laws given $\Ftheta$ are their disintegrations
over $(G,P)$. The two measures agree on the cylinder sets $\{x|_F=w\}$ (a $\pi$-system
generating $\mathscr B_{X_{\bB}}$) hence the fibre measures agree $P$-almost surely. 
Integrating, we obtain $\nu=\kappa$. Conversely $\kappa$ attains the bound, since its allowed
coordinates are fair coins independent of past and phase, giving
$h_\kappa(S)=h_\kappa(S\mid\Ftheta)=D$.
\end{proof}

\section{Second proof of the rigidity: induce-and-split}\label{sec:induce}

We now give a second, structurally different proof of
Proposition~\ref{prop:rigidity}. It realises Peckner's induce-and-split
scheme~\cite{Pe} for $\OK\cong\Z^d$. Conceptually, Lemma~\ref{lem:forbidden}
presents $(X_{\bB},\nu,S)$ over the base $(G,P,T)$ as a measurable skew product,
as follows. For $x\in Y$, associate to $x$ the pair
$(\theta(x),\operatorname{supp}x)$. By Lemma~\ref{lem:forbidden},
$\operatorname{supp}x\subseteq A_{\theta(x)}$, so the pair determines $x$
(forbidden positions carry $0$'s automatically), and $x\mapsto
(\theta(x),\operatorname{supp}x)$ is a Borel bijection, modulo $\nu$-null sets,
between $Y$ and the bundle $\{(g,\xi):\ g\in G,\ \xi\subseteq A_g\}$, whose
fibre over a phase $g$ is the space $\{0,1\}^{A_g}$ of subsets of the allowed
set. In these coordinates the shift acts by $(g,\xi)\mapsto(T_ag,\xi-a)$,
which preserves the bundle since $A_{T_ag}=A_g-a$: a rotation moves the base,
and the fibre is carried along. Under this identification, $\kappa$ is the
fibrewise fair-coin measure. Note that the skew product is \emph{not} a direct
product: the fibre space $\{0,1\}^{A_g}$ itself varies with $g$, and
straightening this dependence is precisely the role of Peckner's inducing in
dimension one, and of the cube blocking below. In one dimension Peckner induces this
skew product to a first-return system on which it \emph{splits} as a direct
product of the (zero-entropy, uniquely ergodic) base with a full shift, whose
maximal measure is unique. For $d\ge2$ there is no first-return map; we replace
the induction by an \emph{exact tiling} of $\OK$ by cubes and the ``split''
reappears as the \emph{conditional independence of the tiles} forced at the
entropy maximum. Throughout, $\nu$ is $S$-invariant with $\nu(Y)=1$ and
$\theta_*\nu=P$, as in Proposition~\ref{prop:rigidity}.

\paragraph{Cube tilings.} Fix the isomorphism $\iota\colon\Z^d\to\OK$
of~\cite{ADK}. For $N\ge1$ put $Q_N:=\iota(\{0,\dots,N-1\}^d)$ and
$\Lambda_N:=\iota(N\Z^d)\le\OK$, a finite-index subgroup with
$[\OK:\Lambda_N]=|Q_N|=N^d$ and exact tiling
$\OK=\bigsqcup_{c\in\Lambda_N}(Q_N+c)$. Write $S^{(N)}:=(S_c)_{c\in\Lambda_N}$ for
the restricted $\Lambda_N$-action and
$\mathcal P_{Q_N}:=\bigvee_{a\in Q_N}S_a^{-1}\mathcal P$ for the cube
super-partition, whose atoms are the patterns $x|_{Q_N}\in\{0,1\}^{Q_N}$.

The first lemma states that no relative entropy is created or lost by
blocking: if we observe the system only along the sublattice $\Lambda_N$, but
read at each observation the entire pattern filling the corresponding cube,
then the relative entropy of this coarser-in-time, richer-in-alphabet process
is exactly $|Q_N|$ times the relative entropy per site of the original action.

\begin{lemma}[Blocking identity]\label{lem:block}
For every $N\ge1$,
$\;|Q_N|\,h_\nu(S\mid\Ftheta)=h_\nu\big(S^{(N)},\mathcal P_{Q_N}\mid\Ftheta\big).$
\end{lemma}
\begin{proof}
Since $\mathcal P$ generates for $S$ and $Q_N$ tiles $\OK$ under $\Lambda_N$,
$\bigvee_{c\in\Lambda_N}S_c^{-1}\mathcal P_{Q_N}=\bigvee_{a\in\OK}S_a^{-1}\mathcal P$,
so $\mathcal P_{Q_N}$ generates for $S^{(N)}$. If $(C_n)$ is a F\o lner sequence in
$\Lambda_N$, then $E_n:=\bigsqcup_{c\in C_n}(Q_N+c)$ is F\o lner in $\OK$ with
$|E_n|=|Q_N|\,|C_n|$, and $\bigvee_{c\in C_n}S_c^{-1}\mathcal P_{Q_N}
=\bigvee_{a\in E_n}S_a^{-1}\mathcal P$. Hence
\begin{equation}\label{eq:blocking}
h_\nu(S^{(N)},\mathcal P_{Q_N}\mid\Ftheta)
=\lim_n\frac{1}{|C_n|}H_\nu\Big(\textstyle\bigvee_{a\in E_n}S_a^{-1}\mathcal P\,\Big|\,\Ftheta\Big)
=|Q_N|\,h_\nu(S,\mathcal P\mid\Ftheta),
\end{equation}
and $h_\nu(S,\mathcal P\mid\Ftheta)=h_\nu(S\mid\Ftheta)$ by \ref{itm:E6},
since $\mathcal P$ generates.
\end{proof}

\begin{lemma}[Two entropy inequalities]\label{lem:twoineq}
For $c\in\Lambda_N$ define the \emph{super-symbols}
$\xi_c(x):=(S_cx)|_{Q_N}\in\{0,1\}^{Q_N}$, so that $\xi_c$ carries the same
information as $x|_{Q_N+c}$ and $(\xi_c)_{c\in\Lambda_N}$ is a
$\Lambda_N$-stationary process relative to the $S$-invariant $\Ftheta$. Then:
\begin{enumerate}[label=\textup{(\alph*)},leftmargin=2.2em]
  \item one has $h_\nu(S^{(N)},\mathcal P_{Q_N}\mid\Ftheta)\le H_\nu(\xi_0\mid\Ftheta)$,
  with equality if and only if the family $(\xi_c)_{c\in\Lambda_N}$ is
  conditionally independent given $\Ftheta$;
  \item one has $H_\nu(\xi_0\mid\Ftheta)\le N^d D$ \textup{(bits)}, with equality if and
  only if, conditionally on $\Ftheta$, the symbol $\xi_0$ is uniformly distributed
  on $\{0,1\}^{A_g\cap Q_N}$.
\end{enumerate}
\end{lemma}
\begin{proof}
(a) This is the relative form of the fact that the mean entropy of a stationary
process is at most its single-symbol entropy, with equality exactly in the
conditionally independent case; see~\cite{OW,WZ}.\footnote{A
self-contained argument: for finite $V\subseteq\Lambda_N$ set
$H(V):=H_\nu\big(\bigvee_{c\in V}\sigma(\xi_c)\mid\Ftheta\big)$. Relative
stationarity and \ref{itm:E4} give $H(V+c)=H(V)$ and 
$H_\nu(\xi\vee\zeta\mid\Ftheta)\le H_\nu(\xi\mid\Ftheta)+H_\nu(\zeta\mid\Ftheta)$, yielding $H(V\cup V')\le H(V)+H(V')$ for disjoint $V,V'$. Let
$V_k\subseteq\Lambda_N$ be the cube of side $k$, so $(V_k)$ is F\o lner in
$\Lambda_N$ and $h:=h_\nu(S^{(N)},\mathcal P_{Q_N}\mid\Ftheta)=\lim_{k\to\infty}H(V_k)/|V_k|$.
Since $V_{mk}$ is tiled exactly by $m^d$ translates of $V_k$, subadditivity and
stationarity give $H(V_{mk})\le m^dH(V_k)$; computing the limit along
$(V_{mk})_{m\ge1}$ yields $h\le H(V_k)/|V_k|$ for every $k$, and in particular
$h\le H(V_1)=H_\nu(\xi_0\mid\Ftheta)$. Suppose $h=H(V_1)$. For each $k$,
$H(V_k)\ge|V_k|\,h=|V_k|\,H(V_1)$, while tiling $V_k$ by singletons gives
$H(V_k)\le\sum_{c\in V_k}H_\nu(\xi_c\mid\Ftheta)=|V_k|\,H(V_1)$; hence
$H(V_k)=\sum_{c\in V_k}H_\nu(\xi_c\mid\Ftheta)$. Now, by the chain rule \ref{itm:E3},
$H_\nu(\xi\vee\zeta\mid\Ftheta)=H_\nu(\xi\mid\Ftheta)+H_\nu(\zeta\mid\sigma(\xi)\vee\Ftheta)$,
and $H_\nu(\zeta\mid\sigma(\xi)\vee\Ftheta)=H_\nu(\zeta\mid\Ftheta)$ if and only
if the conditional mutual information $I_\nu(\xi;\zeta\mid\Ftheta)$ vanishes,
i.e.\ $\xi$ and $\zeta$ are conditionally independent given $\Ftheta$ (the
equality case of \ref{itm:E4}). Iterating
along an enumeration of $V_k$, the additivity of $H(V_k)$ is equivalent to the
conditional independence of $(\xi_c)_{c\in V_k}$ given $\Ftheta$. Since every
finite subset of $\Lambda_N$ lies in a translate of some $V_k$, equality for all
$k$ amounts to conditional independence of the whole family; the converse
direction (independence implies additivity, hence $h=H(V_1)$) is immediate.}

(b) By Lemma~\ref{lem:forbidden} the coordinates of
$\xi_0$ indexed by $Q_N\setminus A_g$ are $\Ftheta$-measurably $0$, so given
$\Ftheta$ the symbol $\xi_0$ takes values in $\{0,1\}^{A_g\cap Q_N}$; therefore
(applying \ref{itm:E1}  conditionally)
\begin{equation}\label{eq:xi-bound}
H_\nu(\xi_0\mid\Ftheta)\le\mathbb E_P\big|A_g\cap Q_N\big|
=\sum_{a\in Q_N}P(\{a\in A_g\})=N^dD,
\end{equation}
where $P(\{a\in A_g\})=P\big(\bigcap_\ell\{g_\ell\neq-a\}\big)
=\prod_{\bb\in\bB}(1-1/N(\bb))=D$ for every $a\in\OK$, the coordinate events
being independent under the product measure $P$,
with equality precisely when the conditional law is uniform on that cube of
allowed coordinates.
\end{proof}

\begin{proof}[Second proof of Proposition~\ref{prop:rigidity}]
Fix $N\ge1$. Combining Lemmata~\ref{lem:block} and~\ref{lem:twoineq},
\begin{equation}\label{eq:chain}
|Q_N|\,h_\nu(S\mid\Ftheta)=h_\nu(S^{(N)},\mathcal P_{Q_N}\mid\Ftheta)
\le H_\nu(\xi_0\mid\Ftheta)\le N^dD=|Q_N|\,D,
\end{equation}
so $h_\nu(S\mid\Ftheta)\le D$.

Suppose now $h_\nu(S\mid\Ftheta)=D$. Then both inequalities in \eqref{lem:twoineq} are equalities, for
every $N$. By Lemma~\ref{lem:twoineq}(b), conditionally on $\Ftheta$ the
super-symbol $\xi_0$ is uniform on $\{0,1\}^{A_g\cap Q_N}$; by relative
stationarity (the conditional law of $\xi_c$ at phase $g$ equals the
conditional law of $\xi_0$ at phase $T_cg$, and $A_{T_cg}=A_g-c$) each
$x|_{Q_N+c}$ is, conditionally on $\Ftheta$, uniform on
$\{0,1\}^{A_g\cap(Q_N+c)}$. By Lemma~\ref{lem:twoineq}(a) the family
$(\xi_c)_{c\in\Lambda_N}$ is conditionally independent given $\Ftheta$. Because the
cubes $\{Q_N+c\}_{c\in\Lambda_N}$ tile $\OK$, these two facts together say that,
conditionally on $\Ftheta$, the coordinates $(x_a)_{a\in A_g}$ are independent
fair coins while $x_a=0$ for $a\notin A_g$. This is exactly the conditional law of
$\kappa$ given $\Ftheta$; since $\theta_*\nu=P=\theta_*\kappa$, uniqueness of
disintegration yields $\nu=\kappa$.

Conversely, as in the first proof of Proposition~\ref{prop:rigidity}, $\kappa$ attains the bound, i.e.
$h_\kappa(S)=h_\kappa(S\mid\Ftheta)=D$.
\end{proof}

\begin{remark}
For $\OK=\Z$ the subgroups $\Lambda_N=N\Z$ recover Peckner's induced first-return
structure~\cite{Pe}, and Lemma~\ref{lem:twoineq}(a) is precisely the statement
that the maximal measure splits as a direct product over the base; thus
Lemmata~\ref{lem:block}--\ref{lem:twoineq} are the exact-tiling form of
induce-and-split, with the periodic cube tiling playing the role of the
first-return map and the conditional independence of tiles playing the role of
the direct-product splitting. 
Finally, although both proofs ultimately rest on the maximality of
entropy of a fair-coin process, they are organised differently: the first proof
(Section~\ref{sec:relent}) maximises \emph{coordinate by coordinate} along a
fixed order, whereas this proof maximises \emph{block by block} and reads off the
relatively independent structure from the equality case of subadditivity \ref{itm:E4}.
\end{remark}

\section{Proof of the main theorem and consequences}\label{sec:proof}

\begin{proof}[Proof of Theorem~\ref{thm:main}]
\emph{Existence.} By \ref{itm:F5}, $X_{\bB}$ has a measure of maximal entropy; and by either
proof of Proposition~\ref{prop:rigidity}, $\kappa$ satisfies
$h_\kappa(S)=h_\kappa(S\mid\Ftheta)=D=\htop(X_{\bB})$ (using
Proposition~\ref{prop:abramov} and \ref{itm:F1}), so $\kappa$ is one.

\emph{Uniqueness.} Let $\nu$ be a measure of maximal entropy. By \ref{itm:F5},
almost every ergodic component of $\nu$ is again maximal,
so it suffices to show every \emph{ergodic} maximal $\nu$ equals $\kappa$. By \ref{itm:F2}
such a $\nu$ is concentrated on $Y$; by Proposition~\ref{prop:haar},
$\theta_*\nu=P$; by Proposition~\ref{prop:abramov},
$h_\nu(S\mid\Ftheta)=h_\nu(S)=\htop(X_{\bB})=D$; and by
Proposition~\ref{prop:rigidity}, $\nu=\kappa$. A general measure of maximal
entropy is, by \ref{itm:F5}, the barycentre of ergodic measures of maximal entropy, each
equal to $\kappa$ by the above; hence it equals $\kappa$, and the measure of
maximal entropy is unique.

\emph{Singularity.} The conditional law of $\kappa$ given
$\Ftheta$ is fair coins on $A_g$ and zeros elsewhere. The measures $\kappa$ and $\nu_\eta$ are distinct ergodic measures
(positive versus zero entropy), hence mutually singular.
\end{proof}

\begin{proof}[Proof of Corollary~\ref{cor:transfer}]
By~\cite{ADK}, the $k$-free number-field system is a special case of the present
setting, and the $k$-free and $\bB$-free lattice-point systems are topologically
conjugate to instances of it. Intrinsic ergodicity is invariant under topological conjugacy, so it transfers
from Theorem~\ref{thm:main} to all these systems.
\end{proof}

The uniqueness in Theorem~\ref{thm:main} upgrades, by a standard argument, to
an equidistribution statement: the local statistics of a \emph{uniformly
random} admissible pattern on a large box are those of $\kappa$. For a finite
set $E\subseteq\OK$ and $w\in\{0,1\}^E$, write
$[w]:=\{x\in X_{\bB}:\ x|_E=w\}$ for the corresponding cylinder set, and let
$\mathcal L_n:=\{x|_{F_n}:\ x\in X_{\bB}\}\subseteq\{0,1\}^{F_n}$ be the set of
admissible patterns on the F\o lner box $F_n$.

\begin{corollary}[Equidistribution of uniformly random admissible patterns]\label{cor:equid}
For every finite $E\subseteq\OK$ and every $w\in\{0,1\}^E$,
\begin{equation}\label{eq:equid}
\lim_{n\to\infty}\ \frac{1}{|\mathcal L_n|}\sum_{u\in\mathcal L_n}
\frac{\#\big\{a\in F_n:\ E+a\subseteq F_n\ \text{and}\
u(e+a)=w(e)\ \text{for all }e\in E\big\}}{|F_n|}
\;=\;\kappa([w]).
\end{equation}
That is, sampling an admissible pattern on $F_n$ uniformly at random and
reading it in a uniformly placed window, the observed statistics converge to
those of the measure of maximal entropy.
\end{corollary}

\begin{proof}
For each $u\in\mathcal L_n$ fix $x_u\in X_{\bB}$ with $x_u|_{F_n}=u$, let
$\mu_n$ be the uniform measure on $\{x_u:u\in\mathcal L_n\}$, and set
$\nu_n:=|F_n|^{-1}\sum_{a\in F_n}(S_a)_*\mu_n$. Since cylinders are clopen and
$E$ is finite, the left-hand side of \eqref{eq:equid} differs from
$\nu_n([w])$ by at most the proportion of $a\in F_n$ with
$E+a\not\subseteq F_n$, which is $o(1)$ by the F\o lner property. Let $\nu$ be
any weak-$*$ limit point of $(\nu_n)_{n\geq1}$. F\o lner averaging makes $\nu$ an 
$S$-invariant measure. We have the  Misiurewicz-type lower bound\footnote{
(Half of) the proof of the variational principle
for amenable group actions (see \cite{MOP} and \cite[Chapter~9]{KL}) gives
$h_\nu(S)\geq\limsup_{j\to\infty}\frac{1}{|F_{n_j}|}H_{\mu_{n_j}}\!\left(\bigvee_{a\in F_{n_j}}S^{-1}_a\mathcal P\right)$ where $\mathcal P$ is a generating clopen partition and $\nu$ is the weak-$*$ limit of $(\nu_{n_j})_{j\geq1}$. Since distinct patterns lie in distinct atoms of $\bigvee_{a\in F_{n_j}}S^{-1}_a\mathcal P$ and $\mu_{n_j}$ is uniform on the set of points in $X_{\bB}$ realizing the $|\mathcal L_{n_j}|$ distinct admissible patterns, the Shannon entropy inside this limsup equals $\log_2|\mathcal L_{n_j}|$. The limsup in \eqref{eq:misiurewicz} is in fact a limit and  equals the topological entropy as in~\cite{ADK}, thus agrees with the limsup along the subsequence $(n_j)_{j\geq1}$
} 
\begin{equation}\label{eq:misiurewicz}
h_\nu(S)\ \ge\ \limsup_{n\to\infty}\frac{1}{|F_n|}\log_2|\mathcal L_n|
\ =\ \htop(X_{\bB},S).
\end{equation} Hence $\nu$ is a measure of maximal
entropy, so $\nu=\kappa$ by Theorem~\ref{thm:main}. The limit point being
unique, we get that $\kappa$ is the weak-$*$ limit of $(\nu_n)_{n\geq1}$. Finally, since the indicator $\mathbf 1_{[w]}$ is continuous on
$X_{\bB}$, we get that  $\nu_n([w])\to\kappa([w])$ as $n\to\infty$, which is \eqref{eq:equid}.
\end{proof}

\begin{remark}
Theorem~\ref{thm:main} makes precise the coexistence, in $\bB$-free systems, of
order and chaos. Measure-theoretically under $\nu_\eta$ the system is a
zero-entropy rotation; topologically it is proximal with trivial maximal
equicontinuous factor; yet its dynamics is rich enough to carry a unique
positive-entropy maximal measure $\kappa$, a relatively independent (Bernoulli)
extension of the very rotation that is invisible as a topological factor. The
distinguished arithmetic point $\eta$ is generic for the rotation measure, not
for $\kappa$.
\end{remark}

\begin{remark}
It would be natural to determine the full simplex of invariant measures and the
equilibrium states of $(X_{\bB},S)$ in the number-field setting, paralleling the
one-dimensional results of~\cite{DKS,KLW}.
\end{remark}

\section*{Acknowledgements}
The author gratefully acknowledges the support of the Natural Sciences and Engineering Research Council of Canada  through the NSERC Discovery Grants  RGPIN-2022-04330.\\
During the preparation of this work the author used Claude (Anthropic, Opus 4.8 and Fable 5 models, accessed via claude.ai) to draft and refine exposition and LaTeX and to assist with literature and reference searches. All mathematical results, proofs, and conclusions are the author's own. The author has verified the full text and is solely responsible for it.

\bigskip

\noindent\textsc{Department of Mathematics and Statistics, Queen's University,
Kingston, Ontario, Canada.}

\smallskip

\noindent\textit{Email address}: \texttt{francesco.cellarosi@queensu.ca}

\end{document}